\documentclass[12pt]{article}
\usepackage{amsmath, amsthm, amssymb} 
\usepackage{times} 
\usepackage{commath} 
\usepackage{epsfig} 
\usepackage[round]{natbib}                  

\providecommand{\abs}[1]{\lvert#1\rvert}
\providecommand{\norm}[1]{\lVert#1\rVert}
\newcommand{\odd}{\in 2\mathbb{N}\hspace{-1pt}+\hspace{-1pt}1}
\newcommand{\even}{\in 2\mathbb{N}}
\newcommand{\ind}{1\hspace{-2.6mm}{1}} 
\newcommand\independent{\protect\mathpalette{\protect\independenT}{\perp}}
\def\independenT#1#2{\mathrel{\rlap{$#1#2$}\mkern2mu{#1#2}}}

\theoremstyle{plain}
\newtheorem{lemma}{Lemma}
\newtheorem{theorem}{Theorem}
\newtheorem{corollary}{Corollary}
\newtheorem{proposition}{Proposition}

\theoremstyle{definition}
\newtheorem{definition}{Definition}
\newtheorem{example}{Example}

\theoremstyle{remark}

\begin{document}
\title{Differential cumulants, hierarchical models and monomial ideals}
\author{
	Daniel Bruynooghe,\\
	Department of Statistics, London School of Economics,
	London, UK\\
	d.hawellek@lse.ac.uk\\
	Henry P. Wynn\\
	Department of Statistics, London School of Economics,
	London, UK\\
	h.wynn@lse.ac.uk
  }
\markboth{D. Bruynooghe and H.P.Wynn}{Differential cumulants, hierarchical models and monomial ideals}
\maketitle
\begin{abstract}
For a joint probability density function $f_X(x)$ of a random 
vector $X$ the mixed partial derivatives of $\log f_{X}(x)$
can be interpreted as limiting cumulants in an infinitesimally small
open neighborhood around $x$. Moreover, setting them 
to zero everywhere gives independence and conditional independence 
conditions. The latter conditions can be mapped, using an algebraic
differential duality, into monomial ideal conditions. 
This provides an isomorphism between hierarchical models and 
monomial ideals. It is thus shown that certain monomial 
ideals are associated with particular classes of hierarchical models.
\end{abstract}

\section{Introduction}
This paper draws together three areas: 
a new concept of differential 
cumulants, hierarchical models and the theory of monomial ideals in algebra.
The central idea is that for a strictly 
positive density $f_X(x)$ of a $p$-dimensional 
random vector $X$, the mixed partial derivative of 
the log density $g_X(x) = \log f_X(x)$ can be used to express 
independence and conditional independence statements. 
Thus, for random variables $X_1,X_2,X_3 \text{ in } \mathbb{R}$, the condition
\begin{equation}
   \frac{\partial^2}{\partial x_1 \partial x_2}g_{X_1,X_2,X_3}(x_1,x_2,x_3) 
	= 
	0 \text{ for all } (x_1,x_2,x_3) \text{ in } \mathbb{R}^3 \label{eqn: mixed partial derivatives} 
\end{equation}
is equivalent to the conditional independence statement
\begin{equation*}
   X_1 \independent X_2 | X_3.
\end{equation*}
In the next section we show how such mixed partial derivatives
can be interpreted as differential cumulants. Then, 
in section \ref{conditional independence statements}, 
we show how collections of differential 
equations like \eqref{eqn: mixed partial derivatives}
can be used to express independence 
and conditional independence models. 
Section \ref{hierarchical models} shows that, more generally, 
these collections can be used to define hierarchical statistical 
models of exponential form. 

Section \ref{monomial ideals} maps the hierarchical model 
conditions to monomial ideals, which are increasingly
being used within algebraic statistics. 
This isomorphism maps, for example, the mixed partial 
derivative condition \eqref{eqn: mixed partial derivatives} 
to the monomial ideal $<x_1x_2>$ within the polynomial 
ring $k[x_1,x_2,x_3]$. 
The equivalence allows ideal properties
to be interpreted as hierarchical 
model properties, opening up an algebraic-statistical
interface with some potential.

\section{Local and differential cumulants} \label{local and differential cumulants}
This section can be considered as a development
from a body of work on local correlation. 
Good examples are the papers of 
\citet{holland1987}, \citet{jones1996a} and \citet{bairamov2003}.
We draw particularly on \citet{mueller2001}.

Let $X \in \mathbb{R}^p$ be a random vector.
We assume $X$ has a $p+1$ times continuously differentiable density $f_X$.
Once we introduce the concept of differential cumulants, 
we further require $f_X$ be strictly positive.

For $x, k \text{ in } \mathbb{R}^p$ we set 
$
	x^k:=\prod_{i=1}^px_i^{k_i}, \; x!:=\prod_{i=1}^px_i!
$
and
$
	m_k 
	= 
	\mathbb{E}(X^k).
$
Let $M_X:\mathbb{R}^p\longrightarrow \mathbb{R}$ 
and $K_X:\mathbb{R}^p\longrightarrow \mathbb{R}$ 
denote the moment and cumulant generating functions of $X$ respectively.
For a vector $k \text{ in } \mathbb{N}^p$ we set  
\begin{equation*}
	D^k f(x) := \frac{\partial^{\,\norm{k}_1}}{\prod_{i=1}^p \partial x_i^{k_i}} f(x),
\end{equation*}
where $\norm{k}_1:=\sum_{i=1}^p \abs{k_i}$ is the Manhattan norm. 
By convention $D^0f(x):=f(x)$.

The cumulant $\kappa_k$ can be found by evaluating 
$D^k(\log(M_X(t))$ at zero. 
We use the multivariate chain rule 
\citep[given e.g. in][]{hardy2006} stated in Theorem \ref{theorem: Hardy's theorem}. 
At the heart of the chain rule is an identification 
of differential operators with multisets:
\begin{definition}[Multiset, multiplicity, size]
	A multiset $M$ is a set which may hold multiple copies of its elements.
	The number of occurrences of an element is its multiplicity. 
	The multiplicity of a multiset is the vector of multiplicities of its elements, 
	denoted by $\nu_M$. 
	The total number of elements $\abs{M}$\text{ in } $M$ is the size. 
	A multiset which is a set is called degenerate.
\end{definition}
\begin{example}[Partial derivative and multiset]
	The partial derivative
	$\frac{\partial^3}{\partial x_1 \partial x_3^3}f(x)$ 
	has associated multiset $\{1,3,3, 3\}$
	with multiplicity $(1,0,3)$ and size four. 
\end{example}
\begin{definition}[Partition of a multiset]
	Let $I$ be some index set and $(M_i)_{i \text{ in } I}$ be a family of multisets 
	with associated family of multiplicities
	$(\nu_{M_i})_{i \text{ in } I}$.
	A partition $\pi$ of a multiset $M$ is a multiset of multisets $\{(M_i)_{i \in I}\}$
	such that $\nu_M=\sum_{i \in I}\nu_{M_i}$.
	Being a multiset itself, a partition  
	can hold multiple copies of one or more multisets. 
\begin{example}[Partition of a multiset]
	The multiset $\{\{x_1,x_3\}, \{x_1,x_3\}, \{x_3\}\}$ is a partition 
	of $\{x_1,x_1,x_3,x_3,x_3\}$, since $(1,0,1)+(1,0,1)+(0,0,1) = (2,0,3)$.
	In the following, we will use the shorthand $\{x_1x_3|x_1x_3|x_3\}$.
\end{example}
\end{definition}
Associated with a partition $\pi$ of a multiset $M$ is a combinatorial quantity
to which we refer as the collapse number $c(\pi)$. It is defined as
\begin{equation*}
	c(\pi):=\frac{\nu_M!}{\prod_{i\in I}\nu_{M_i}! \nu_{\pi}!}.
\end{equation*}
See \cite{hardy2006} 
for a combinatorial interpretation of $c(\pi)$.
\begin{theorem}[Higher order derivative of chain functions] \label{theorem: Hardy's theorem}
\begin{equation*}
   D^kg(h(x))=\sum_{\pi \in \Pi(k)} c(\pi)
	D^{\abs{\pi}}g(h)
	\prod_{j=1}^{\abs{\pi}}D^{\nu_{M_j}}h(x),
\end{equation*}
where $\Pi(k)$ is the set of all partitions of a multiset with multiplicity $k$
and $M_j$ is the j-th multiset in the partition $\pi$.
\end{theorem}
\begin{proof}
   See \citet{hardy2006}.
\end{proof}

%
%
%
\begin{corollary}[Cumulants as functions of moments] \label{corrolary: cumulants in terms of moments}
Let $\kappa_k$ be the k-th cumulant. Then
\begin{equation}
   \kappa_k = \sum _{\pi \in \Pi(k)} c(\pi)(-1)^{(\abs{\pi}-1)} (\abs{\pi}-1)!
	\prod_{j = 1}^p m_{\nu_{M_j}}\label{eqn: cumulants in terms of moments}.
\end{equation}
\end{corollary}
\begin{proof}
   Set $g(h)=\log(h)$, \; $h(t)=M_X(t)$ and evaluate at $t=0$.
\end{proof}
\begin{example}[Partial derivative]
	Consider the partial derivative $\frac{\partial^3}{\partial x \partial z^2}g(h(x,y,z)$.
	The associated multiset is $\{1,3,3\}$ with partitions
	$\{133\}$,
	$\{13|3\}$,
	$\{1|33\}$,
	$\{1|3|3\}$.
   The multivariate chain rule tells us that
	\begin{align*}
		D^{102}g(h(x,y,z))
		& =
		DgD^{102}h\\
		& \quad +
		2D^2gD^{101}hD^{001}h\\
		& \quad +
		D^2gD^{100}hD^{002}h\\
		& \quad +
		D^3gD^{100}h(D^{001}h)^2,
	\end{align*}
	where function arguments have been suppressed on the right hand side for better readability.
   In particular we may conclude that 
	\begin{align*}
		\kappa_{102}  
		=
		m_{102}
		- 
		2 m_{101}m_{001}
		- 
		m_{100}m_{002}
		+
		m_{100}m^2_{001}.
	\end{align*}
\end{example}
The expression for cumulants in terms of moments is particularly 
simple in what we shall call the square-free case, that is for 
cumulants $\kappa_k$, whose index vector $k$ is binary.
In that case, the multiset associated with $k$ is degenerate
and $c(\pi)=1$.
Equation \eqref{eqn: cumulants in terms of moments} simplifies to
\begin{equation*}
   \kappa_k = \sum _{\pi \in \Pi(k)} (-1)^{(\abs{\pi}-1)} (\abs{\pi}-1)!
	\prod_{j = 1}^p m_{\nu_{M_j}}.
\end{equation*}
In this form it is often stated and derived via the classical 
Faa Di Bruno formula applied to an exponential function 
followed by a Moebius inversion \citep[see e.g.][]{barndorff1989}. 

Local analogues to moments and cumulants can be derived as one considers
their limiting counterparts in the neighborhood 
of a fixed point $\xi \text{ in } \mathbb{R}^p$, 
an idea proposed by \citet{mueller2001}.
This section derives formulae for local moments and cumulants and 
local moment generating functions 
provided its global counterpart exists.

For a strictly positive edge length $\epsilon \text{ in } \mathbb{R}_+$, 
let $A(a, \epsilon):=[\xi-\frac{\epsilon}{2}, \xi+\frac{\epsilon}{2}]^p$ 
denote the hyper cube centralized at $\xi$.  
Let $\abs{A} =\epsilon^p$ denote its volume.
The density of the random variable $X\text{ in } \mathbb{R}^p$ 
conditional on being in $A$ is given by 
\begin{equation*}
	f^A_X(x) = \frac{f_X(x)\ind_A(x)}{pr(X\in A)},
\end{equation*}
where $\ind_A(x)$ is the indicator function which returns unity if 
$x$ is in $A$ and zero otherwise. 
The conditional moments about $\xi$ are denoted by 
\begin{equation*}
	m^A_k = \mathbb{E}\bigg(\prod_{i=1}^p (X_i-\xi_i)^{k_i} \sVert[2] X \in A\bigg).
\end{equation*}
Let $2\mathbb{N}$ and $2\mathbb{N}\hspace{-3pt}+\hspace{-3pt}1$
denote the set of positive even and odd integers respectively.
For symmetry reasons, even and odd orders of individual 
components have different effects on local moments,
which motivates the following definition:
\begin{equation*}
   \norm{\alpha}^+_1:=\norm{\alpha}_1+\sum_{i=1}^p \ind \;(\alpha_i \odd).
\end{equation*}
$\norm{\cdot}^+_1$ increments the total sum of the components of a vector 
by one additional unit for each odd component (it is 
not to be interpreted as a norm).

\begin{theorem}[Local moments]\label{theorem: local moment}
Let $X \text{ in } \mathbb{R}^p$ be an absolutely continuous random 
vector with density $f_X$ which is $p$ times differentiable 
in $\xi \text{ in } \mathbb{R}^p$.
Let $k \text{ in } \mathbb{N}^p$ determine the order of moment.
Then, for $\abs{A}$ sufficiently small, 
$X$ has local moment
\begin{align}
	m^A_k
	& = 
	r(\epsilon, k)
	\bigg(
		\frac
		{
			D^
			{
				\alpha
			}
			f_X(\xi)
		}
		{
			f_X(\xi)
		} 
		+O(\epsilon^2)
	\bigg),
	\label{eqn: result local moment}
\end{align}
where
$
r(\epsilon, k):=
	\epsilon^
	{
		\norm{k}_1^+
	}
	{\displaystyle \prod_{\substack{i=1, \\ k_i \even}}^p}
	\frac
	{
	   1
	}
	{
		k_i+1
	}
	{\displaystyle \prod_{\substack{i=1, \\ k_i \odd}}^p}
	\frac
	{
		1
	}
	{
		k_i+2
	}
$
and
$
   \alpha:= {\displaystyle \sum_{\substack{i=1, \\ k_i \odd}}^p} e_i
$
.
\end{theorem}

\begin{proof}
Consider
\begin{align} 
	m^A_k
	& = 
	\frac{
		\int_A\prod_{i=1}^p(x_i-\xi_i)^{k_i} f_X(x) dx 
	}
	{
		\int_A f_X(x) dx 
	}
   \label{eqn: definition local moment}
\end{align}
Approximate $f_X$ through its $p$-th order Taylor expansion, 
integrate \eqref{eqn: definition local moment} term by term 
and exploit the point symmetry of odd order terms 
about the origin. 
\end{proof}

\begin{example}[Local moment $m_{120}$]
Consider a tri-variate random variable $X$ 
with local moment 
$
   m_{120}^A=E((X_1-\xi_1)(X_2-\xi_2)^2|X\in A).
$
Then 	
$
	r(\epsilon, k)=\frac{\epsilon^4}{9},
   \alpha:= (1,0,0)'
$
and we obtain	
\begin{equation*}
	m_{120}^A=\frac{\epsilon^4}{9}\frac{\partial f(x_1,x_2,x_3)}{\partial x_1}+O(\epsilon^6).
\end{equation*}
\end{example}

A natural way to extend the concept of a local moment 
is to consider the limiting case when $\epsilon \rightarrow 0$.
This leads to our definition of differential moments.
\begin{definition}[Differential moment]
The differential moment of an absolutely continuous random vector 
$X \text{ in } \mathbb{R}^p$ in $\xi \text{ in } \mathbb{R}^p$ is defined as: 
\begin{align*}
	m^\xi_k
   := 
	\lim_{\epsilon \rightarrow 0}
	\frac
	{
		m^A_k
	}
	{
		r(\epsilon, k)
	}.
\end{align*}
\end{definition}
\begin{corollary}[Differential moment]
For a differential moment of order 
$k \text{ in } \mathbb{N}^p$ in $\xi \text{ in } \mathbb{R}^p$ it holds that
\begin{align*}
	m^\xi_k
	=
	\frac
	{
		D^
		{
			\alpha
		}
		f_X(\xi)
	}
	{
		f_X(\xi)
	} 
	.
\end{align*}
\end{corollary}
\begin{proof}
	This follows from Theorem \ref{theorem: local moment} upon taking the limit as $\epsilon \rightarrow 0$.
\end{proof}
From \eqref{eqn: result local moment} it is clear 
that the choice of $\alpha$ in the derivative $D^{\alpha} f_X$
depends only on 
the pattern of odd and even components of the moment.
To be precise, $\alpha$ holds a unity corresponding 
to odd components and a zero corresponding to even component entries.
Consequently, the differential moment $m^\xi_k$ depends on $k$ only via 
the pattern of odd and even values. 

This suggests defining an 
equivalence relation on $\mathbb{N}^p\times\mathbb{N}^p$:
For $u, k \in \mathbb{N}^p$ set
$u \sim_{m} k \iff m_{u_1\cdots u_p} = m_{k_1\cdots k_p}$. 
The relation $\sim_{m}$ partitions the product space $\mathbb{N}^p\times\mathbb{N}^p$
into $2^p$ equivalence classes of same differential moments.
The graph corresponding to $\sim_m$ is depicted in 
Figure \ref{fig: equivalence graph} for the bivariate case.
The axes give the order of the moment for the two components.
Different symbols represent different equivalence classes. 
For instance, $(2,2) \sim_m (4,2)$, since $m^\xi_{22}=m^\xi_{42}$.
Note that $u \sim_{m} k \iff \norm{u-k}_1 \even $. 

\begin{figure}
	\begin{center}
		\epsfig{file= 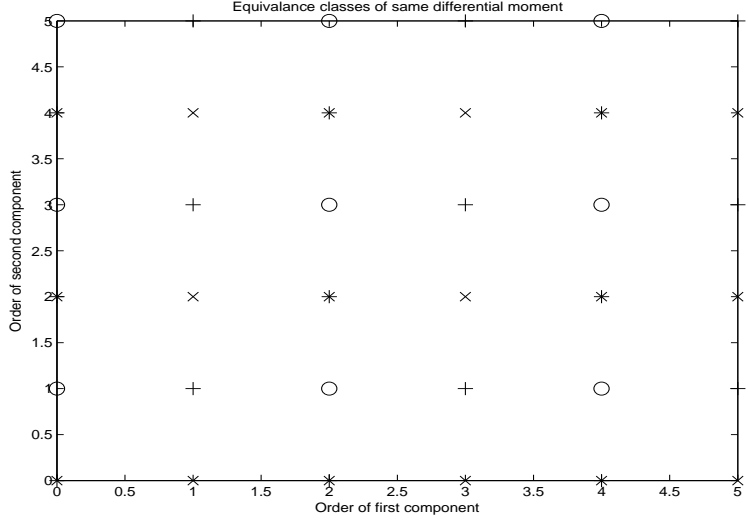,width=10cm,height=7cm}
	\end{center}
	\caption[Equivalent classes with same differential moments]
	{
	   Graph of the equivalence classes induced by
		$\sim_m$ (bivariate case).  
		Each equivalence class is depicted with a different symbol. 
	}
	\label{fig: equivalence graph}
\end{figure}

Similarly to local moments, for any measurable set $A$ we can define a local moment generating function:
\begin{equation*}
   M_X^A(t):=\mathbb{E}(e^{t'X}|X\in A).
\end{equation*}
Being a conditional expectation, 
it exists if $M_X$ exists. 
We have the following expansion: 

\begin{align*}
	M^A_X(t)
	& =
	\frac{1}{pr(X \in A)}
	\int_A
	e^{\sum_{i=1}^pt_ix_i}f_X(x)dx\\
	& =
		%
	1
	+ 
	\sum_{i=1}^p t_i\eval{\frac{\partial f_X(x)}{\partial x_i}}_{x=\xi} \bigg(\frac{\epsilon^2}{3f_X(\xi)}+O(\epsilon^4)\bigg) \\
	& + 
   \sum_{i=1}^p t_i^2\eval{\frac{\partial^2 f_X(x)}{\partial x_i^2}}_{x=\xi} \bigg(\frac{\epsilon^2}{6f_X(\xi)}+O(\epsilon^4)\bigg)\\
	& + 
   \sum_{i=1}^p\sum_{j>i}^p t_it_j\eval{\frac{\partial^2 f_X(x)}{\partial x_i\partial x_j}}_{x=\xi}
	\bigg(
		\frac{\epsilon^4}{9f_X(\xi)}+O(\epsilon^6)
	\bigg)
	+O(\epsilon^4\norm{t^3}).
\end{align*}
The local moments can be computed from the local moment generating function via 
differentiation to appropriate order and evaluation at $t=0$.
The natural logarithm of the local moment generating function defines the 
local cumulant generating function $K^A_X(t):\mathbb{R}^p\longrightarrow \mathbb{R}$: 
\begin{equation*}
	K^A_X(t):=\log(M^A_X(t)).
\end{equation*}
\begin{corollary}[Local cumulants]\label{corollary: local cumulants}
Under the conditions of Theorem \ref{theorem: local moment} it holds
for the local cumulants that
\begin{align*}
   \kappa_k^A
	=
   \sum _{\pi \in \Pi(k)} 
	c(\pi) (-1)^{(\abs{\pi}-1)} (\abs{\pi}-1)!
	\prod_{j=1}^{\abs{\pi}}
	r(\epsilon, \nu_{M_j})
	\bigg(
		\frac
		{
			D^{\alpha_j}f_X(\xi)	
		}
		{
			f_X(\xi)
		}
		+ O(\epsilon^2)
	\bigg)
	,
\end{align*}
where $\alpha_j$ is a function of the partition $\pi$ and defined as 
\begin{align*}
   \alpha_j:={\displaystyle \sum_{i=1}^p e_i \ind{\bigg(\nu_{M_j}(i) \odd\bigg)}},
\end{align*}
that is, $\alpha_j$ is binary and holds ones corresponding to odd elements of $\nu_{M_j}$.
Furthermore,
\begin{align*}
	r(\epsilon, \nu_{M_j})
	&:=
	\epsilon^
	{
		\norm{\nu_{M_j}}_1^+
	}
	{\displaystyle \prod_{\substack{i=1, \\ \nu_{M_j}(i) \even}}^p}
	\frac
	{
	   1
	}
	{
		\nu_{M_j}(i)+1
	}
	{\displaystyle \prod_{\substack{i=1, \\ \nu_{M_j}(i) \odd}}^p}
	\frac
	{
		1
	}
	{
		\nu_{M_j}(i)+2
	}
	.
\end{align*}
\end{corollary}
\begin{proof}
Combine the chain rule and Theorem \ref{theorem: local moment}.
\end{proof}

Similarly to differential moments we can define differential cumulants at $\xi$.
Two different ways of doing so are natural.
First, taking the limiting quantity of the local cumulants as 
$\epsilon \rightarrow 0$ or, second,
taking the series of differential moments and requiring 
that the mapping between moments and cumulants is preserved 
which is induced through the ex-log relation of the 
associated generating functions, see 
also the discussion in \cite[page 62]{mccullagh1987}.

As demonstrated below, the two quantities just described 
differ in general and
coincide only in the square-free case. 
In order to retain the intuitive and familiar relation between 
cumulants and moments, we define differential cumulants in 
terms of differential moments.
\begin{definition}[Differential cumulant]
For an index vector $k \text{ in } \mathbb{N}^p$, the differential cumulant
in $a \text{ in } \mathbb{R}^p$ is defined as
\begin{align*}
   \kappa_k^a
	:=
   \sum _{\pi \in \Pi(k)} 
	c(\pi) (-1)^{(\abs{\pi}-1)} (\abs{\pi}-1)!
	\prod_{i=1}^{\abs{\pi}}
	m^a_{\nu_{M_i}}.
\end{align*}
\end{definition}
We are now in a position to state the main 
result of this section, namely that mixed partial 
derivatives of the log density can be interpreted 
as differential cumulants. 
\begin{lemma}[Differential cumulant] \label{lemma: Differential cumulants}
For a differential cumulant in $\xi \text{ in } \mathbb{R}^p$
of order $k \text{ in } \mathbb{N}^p$  it holds that
\begin{align*}
	\kappa^\xi_k
	=
	D^
	{
		\alpha
	}
	\log
	(
		f_X(\xi)
	)
	,
\end{align*}
where
$
   \alpha:= {\displaystyle \sum_{\substack{i=1, \\ k_i \odd}}^p} e_i
$
projects odd elements of $k$ onto one and even elements of $k$ onto zero. 
\end{lemma}

\begin{proof}
Apply the chain rule to 
$
	D^
	{
		\alpha
	}
	\log
	(
		f_X(\xi)
	)
$.
\end{proof}
This is a multivariate generalization of the local dependence function 
introduced by \cite{holland1987}.
The next theorem relates differential cumulants to 
the limit of local cumulants.
\begin{theorem}[Differential and limiting local cumulant]
A differential cumulant $\kappa^\xi_k$ equals the limit
of the local cumulant 
$
	lim_{\epsilon \rightarrow 0} \frac{1}{r(\epsilon,k)}\kappa^A_k
$
if and only if $k$ is binary, i.e. $\kappa_k$ is a square-free cumulant. 
\end{theorem}
\begin{proof}
First, let $k \in \{0,1\}^p$ be binary and $\pi=\{(M_j)_{1 \leq j \leq \abs{\pi}}\}$ 
be a partition of the lattice corresponding to $k$. 
One can show that 
$r(\epsilon, k) = \prod_{j=1}^{\abs{\pi}}r(\epsilon, \nu_{M_{j}})$. 
With that
\begin{align}
	\frac{1}{r(\epsilon,k)}\kappa^A_k
	& =
   \sum _{\pi \in \Pi(k)}
	(-1)^{(\abs{\pi}-1)} (\abs{\pi}-1)!
	\prod_{\substack{j=1, \\ M_j \in \pi}}^{\abs{\pi}}
	\frac
	{
		D^{\nu_{M_j}}f_X(\xi)
	}
	{
		f_X(\xi)
	}
   +O(\epsilon^2). \label{taking local cumulants to zero}
\end{align}
Now take limits as $\epsilon \rightarrow 0$ to obtain
$
	lim_{\epsilon \rightarrow 0} 
	\frac{1}{r(\epsilon,k)}\kappa^A_k
	= 
	\kappa_k^\xi.
$

Conversely, suppose $k$ is not binary. 
Express
$
	\kappa^A_k
$
as a linear combination of local moments. 
Consider the degenerate partition $\pi$, which
holds only one multiset $M$ with multiplicity $\nu_M=k$.
The quantity associated with $\pi$ converges to
$
	c
	\frac
	{
		D^kf_X(\xi)
	}
	{
		f_X(\xi)
	} \label{wrong convergence}
$
for some constant $c \text{ in } \mathbb{R}$.
$k$ not being binary, this cannot be a differential moment, 
which are proportional to
$D^\alpha f_X(\xi)$ for some binary $\alpha$.
Differential cumulants are linear combinations
of differential moment products only. 
Hence 
$
	\kappa^A_k
$
does not converge to a differential cumulant.
\end{proof}

Of particular interest to us are differential 
cumulants which vanish everywhere. 
We refer to them as zero-cumulants. 
Writing $g=\log f_X$, 
we shall usually write $D^\alpha g=0$
to denote the zero-cumulant associated with $\alpha$
in the understanding that this holds for all $x$. 

The next section shows that
sets of zero cumulants are isomorphic 
to conditional independence statements.
As a consequence of lemma \ref{lemma: Differential cumulants} 
zero-cumulants are invariant 
under diagonal transformations of the random vector $X$. 
In particular, they are not affected 
by the probability integral transformation and hence any 
result below holds also true for the copula density of $X$.
 
\section{Independence and conditional independence} \label{conditional independence statements}
From now on, we shall assume 
that $f_X$ is strictly positive everywhere. 
Sets of zero-cumulants are equivalent to
conditional and unconditional dependency structures. 
\begin{proposition}[Independence in the bivariate case] \label{prop: independence bivariate case}
   Let $X \text{ in } \mathbb{R}^2$.
	Then $X_1 \independent X_2 \iff \kappa_{11}^x=0 \quad  \text{ for all } x \text{ in } \mathbb{R}^2$.
\end{proposition}
\begin{proof}
\begin{align*}
	0 = \kappa_{11}^x  
	= 
	\frac{\partial^2}{\partial x_1\partial x_2}\log(f_{X_1,X_2}(x_1,x_2) 
	\iff
	f_{X_1,X_2}(x_1,x_2) & = e^{h_{1}(x_1)+h_{2}(x_2)} 
\end{align*}
for some functions $h_1,h_2:\mathbb{R}\rightarrow \mathbb{R}$.
\end{proof}
In the multivariate case, we can express conditional 
independence of any pair given the remaining variables
through square free differential cumulants.
\begin{proposition}[Conditional independence of two random variables] \label{prop: conditional independence pair of r.v.}
   Let $X \text{ in } \mathbb{R}^p$.
	Then 
\begin{equation*}
	X_i \independent X_j | X_{-ij} \iff \kappa_{k}^x=0 \quad  \text{ for all } x \text{ in } \mathbb{R}^p,
\end{equation*}
	where $$X_{-ij}:=(X_1,...,X_{i-1},X_{i+1},...,X_{j-1},X_{j+1},...,X_p)$$ and
	$k=e_i+e_j ,\;\;(i, j) \in \{1,...,p\}^2, \;\;i \neq j $.
\end{proposition}
\begin{proof}
By analogy with the bivariate case.
\end{proof}
Setting several square-free differential cumulants to zero 
simultaneously allows us to express conditional independence 
statements. 
\begin{proposition}[Multivariate conditional independence] \label{prop: conditional independence multivariate case}
Given three index sets $I,J,K$ which 
partition $\{1,...,p\}$, let $S =  \{e_i+e_j, i \in I, j \in J\}$. 
Then
\begin{equation*}
   X_I \independent X_J | X_{K} \iff \kappa_{k}^x=0 
   \text{ for all }
   k \in S 
   \text{ and for all }
   x \text{ in } \mathbb{R}^p.
\end{equation*}
\end{proposition}
\begin{proof}
From proposition \ref{prop: conditional independence pair of r.v.} it is clear, 
that this is equivalent to the 
conditional independence statement
\begin{equation*}
   X_I \independent X_J | X_{K} 
	\iff 
	X_i \independent X_j | X_{-ij} \quad  \text{ for all } (i,j) \in I\times J.
\end{equation*}
Sufficiency ($\Rightarrow$) and necessity ($\Leftarrow$) 
are semi-graphoid and graphoid axioms 
referred to as decomposition and intersection respectively. 
Both hold true 
for strictly positive conditional densities
\citep[see for instance][]{cozman2005}.
\end{proof}

Pairwise conditional independence of all pairs is equivalent to 
independence. 
\begin{theorem}[Pairwise conditional independence if and only if independence]
	The random variables $X_1,...,X_n$ are independent 
	if and only if $\kappa_{e_i+e_j}=0 \quad
	 \text{ for all } (i, j) \in \{1,...,n\}^2 , \;\;i \neq j$.
\end{theorem}
\begin{proof}
Sufficiency follows from differentiation of the log-density.
Necessity can be proved by induction on 
the number of variables $n$. 
The statement is true for $n=2$ by proposition \ref{prop: independence bivariate case}.
Let the statement be true for $n$ and let the 
$n+1 \choose 2$ differential cumulants $\kappa_{e_i+e_j}$ vanish, 
where $e_i$ and $e_j$ are unit vectors in $\mathbb{R}^{n+1}$.
Consider $\kappa_{e_1+e_2}=0$. 
Integration with respect to $x_1$ and $x_2$ yields 
\begin{equation} \label{eqn: factor density}
   f_{X_1,...,X_{n+1}}(x_1,...,x_{n+1})
	=
	e^
	{
			h_{1}(x_{-1})
			+
			h_{2}(x_{-2})
	}
\end{equation}
for some functions $h_1:\mathbb{R}^n \longrightarrow \mathbb{R}$ and $h_2:\mathbb{R}^n \longrightarrow \mathbb{R}$.
Now integrate again with respect to $x_1$ to obtain
\begin{equation*}
   f_{X_{-1}}(x_{-1})=
	e^{h_{1}(x_{-1})}
	\int_{\mathbb{R}} e^{h_{2}(x_{-2})} dx_1.
\end{equation*}
The left hand side is an n-dimensional marginal density which 
factorises into $n$ marginals by induction assumption:
$
f_{X_{-1}}(x_{-1})=\prod_{i=2}^{n+1}f_{X_i}(x_i).
$
This allows us to conclude that $h_{1}(x_{-1})$ 
can be split into a sum of two functions, $g_1:\mathbb{R}^{n-1} \longrightarrow \mathbb{R}$ 
and $g_2:\mathbb{R} \longrightarrow \mathbb{R}$, 
where the latter is a function of $x_2$ only, i.e.
$h_{1}(x_{-1}) = g_{1}(x_{-12})+g_{2}(x_{2}).$
Considering \eqref{eqn: factor density} again we see that the density $f_{{X_1},...,{X_{n+1}}}$ factorises 
\begin{equation*} 
   f_{X_1,...,X_{n+1}}(x_1,...,x_{n+1})
	=
	e^
	{ 
		g_{2}(x_{2})
		+
	   g_{1}(x_{-12})
		+
	   h_{2}(x_{-2})
	}
	.
\end{equation*}
Hence $X_2 \independent X_{-2}$ and the density of $X_{-2}$
factorises by induction assumption. 
\end{proof}

\section{Hierarchical models} \label{hierarchical models}

The analysis of the last section makes clear that setting certain mixed
two-way partial derivatives  of $g(x) = \log f_X(x)$ equal to zero,
is equivalent to independence or conditional independence statements.
We can go further and define a generalized hierarchical model using
the same process.

The basic structure of a hierarchical  model can be define via a
simplicial complex. Thus let $\mathcal{N} = \{1, \ldots, p\}$ 
be the vertex set representing the random variables $X_1,...,X_p$. 
A collection
$\mathcal S$ of index sets $J \subseteq \mathcal N$ is a
simplicial complex if it is closed under taking subsets, 
i.e. if $J \text{ in } \mathcal S$ and $K \subseteq J$ then $K
\text{ in } \mathcal S$.

\begin{definition} Given a simplicial complex $\mathcal S$ over an index
	set $\mathcal{N} = \{1, \ldots, p\}$ and an absolutely continuous random 
	vector $X$ a hierarchical model for the joint
	distribution function $f_X(x)$ takes the form:
	\begin{equation*}
		f_X(x) = \exp \left\{ \sum_{J \text{ in } \mathcal S} h_J(x_J)\right\},
	\end{equation*}
	where $h_J:\mathbb{R}^J\longrightarrow \mathbb{R}$ and
	$x_J\text{ in } \mathbb{R}^J$ is 
	the canonical projection of $x \text{ in } \mathbb{R}^p$ onto 
	the subspace associated with the index set $J$.
\end{definition}

This is equivalent to a quasi-additive model
for $g(x) =  \sum_{J \in \mathcal S} h_J(x_J), $
and we also refer to this model for $g(x)$ as being hierarchical. 
It is clear that we may write the model over the maximal cliques only, 
namely simplexes which are not contained in a larger simplex.
In the terminology of \citet{lauritzen1996} we require
$f_X$ be positive and factorise according to $\mathcal{S}$ 
for it to be a hierarchical model with respect to $\mathcal{S}$.

Associated to an index set $K\subseteq \mathcal{N}$ is 
a differential operator $D^k$, where 
$k=\sum_{i\in K} e_i \in \{0,1\}^p$ 
holds ones for every member of $K$ and zeros otherwise. 
In the following, we overload the differential operator 
by allowing it to be superscripted by a set or by a vector.  
Thus, for an index set $K$ we set $D^K:=D^k$ and similarly $\kappa_K^x:=\kappa_k^x$.
$D^K$ returns the differential cumulant $\kappa_K^x$, 
when applied to $g(x)$.
\begin{example}
	Let $K = \{2,4,6\}$. We obtain $k = (0,1,0,1,0,1)$
	and 
	\begin{equation*}
	   D^Kg(x) =\kappa_K^x = \kappa_k^x 
		= 
		\frac{\partial^3}{\partial x_2 \partial x_4 \partial x_6} g(x).
	\end{equation*}
\end{example}
We collect the results of the last section into a comprehensive
statement. First, we define the complementary complex to a
simplicial complex $\mathcal S$ on $\mathcal{N}$.
\begin{definition} Given a simplicial complex $\mathcal S$ on an
	index set $\mathcal{N}$ we define the complementary complex as the collection
	$\bar{\mathcal S}$ of every index set $K$ which is not a member of
	$\mathcal S$.
\end{definition}

Note immediately that $\bar{\mathcal S}$ is closed under unions, 
i.e. $K,K' \text{ in } \bar{\mathcal S} \Rightarrow K
\cup K' \in \bar{\mathcal S}$.
It is a main point of this paper that there is a duality between
setting collections of mixed differential cumulants 
equal to zero and a general
hierarchical model:
\begin{theorem}\label{theorem: hierarchical model and differential cumulants}
	Given a simplicial complex $\mathcal S$ on an index set
	$\mathcal{N}$, a model $g$ is hierarchical, based on $\mathcal
	S$ if and only if all differential cumulants on the 
	complementary complex vanish everywhere, that is
	\begin{equation*}
		 \kappa_K^x =0,\; \mbox{ for all}\; x \text{ in } \mathbb{R}^p \; \mbox{and for all}\;
		 K \text{ in } \bar{\mathcal S}.
	\end{equation*}
\end{theorem}

\begin{proof} 
	First, let $g$ be hierarchical with respect to $\mathcal S$, 
	that is $g$ is a log-density with representation 
	$g(x)= \sum_{J \text{ in } \mathcal S} h_J(x_J)$.
	Then, for $K \text{ in } \bar{\mathcal S}$, the associated 
	differential operator $D^K$ 
	annihilates any term $h_J$ in $g$, 
	since $K \not \subseteq J \text{ for any } J \text{ in } S$.

	Conversely, suppose
	$
		 \kappa_K^x =0 \mbox{ for all}\; x \text{ in } \mathbb{R}^p \; \mbox{and for all}\;
		 K \text{ in } \bar{\mathcal S}.
   $	
	Then, by proposition \ref{prop: conditional independence pair of r.v.},
   $f_X$ is pairwise Markov with respect to $\mathcal{S}$ and hence 
	factorises over maximal cliques of $\mathcal{S}$ by the
	Hammersley-Clifford theorem. 
	The reader is referred to \citet{lauritzen1996} 
	for a detailed discussion of factorization and Markov properties. 
\end{proof}

\section{The duality with monomial ideals} \label{monomial ideals}
The growing area of algebraic statistics makes use of 
computational commutative algebra particularly 
for discrete probability model, notably 
Poisson and multinomial log-linear models. 
Work connecting the algebraic methods to continuous probability
models is sparser although considerable process 
has been made in the Gaussian case. For an overview
see \cite{drton2009}. 
Our link to the algebra is via monomial ideals.

A monomial in $x,...,x_p$ is a product of the form $x^\alpha = \prod_{j=1}^px_j^{\alpha_j}$, 
where $\alpha \text{ in } \mathbb{N}^p$. 
A monomial ideal $I$ is a subset of a polynomial ring $k[x_1,...,x_p]$
such that any $m \in I$ can be written as 
a finite polynomial combination 
$m = \sum_{k \in K} h_k x^{\alpha_k}$, where $h_k \in k[x_1,...,x_p]$
and $\alpha_k \in \mathbb{N}^p$ for all $k \in K$.  
We write $I = <x^{\alpha_1},...,x^{\alpha_K}>$ to express that 
$I$ is generated by the family of monomials $(x^{\alpha_k})_{k \in K}$. 

The full set $M$ of monomials contained in monomial ideal $I$ has
the hierarchical structure:
\begin{equation}
	x^{\alpha} \in M \Rightarrow x^{\alpha + \gamma} \in M, \label{eqn: hierarchical structure}
\end{equation} 
for any index set $\gamma \in \mathbb{N}^p$.
A monomial ideal is square-free if its generators $(x^{\alpha_{k}}
)_{1 \leq k \leq K}$ are square free, i.e. $\alpha_k \in \{0,1\}^p \mbox{ for all }
1 \leq k \leq K$. 

The following discussion shows that there is complete duality
between the structure of square-free monomial ideals and
hierarchical models. Associated with a simplicial complex
$\mathcal S$ is its {\em Stanley-Reisner ideal} $I_{\mathcal S}$. 
This is the ideal generated by all square-free
monomial in the complementary complex $\bar{\mathcal
S}$. For a face $K \in \bar{\mathcal S}$ let 
$m_K(x):=\prod_{k \in K} x_k$ denote the associated square-free 
monomial.
Then 
\begin{equation*}
	I_{\mathcal{S}} = <(m_K)_{K \in \bar{\mathcal{S}}}>.
\end{equation*}

The second step, which is a main point of the paper, is to associate
the differential operator $D^K$ with the monomial $m_K(x)$. We need
only confirm that the hierarchical structure implied by 
\eqref{eqn: hierarchical structure}
is consistent with differential
conditions of Theorem \ref{theorem: hierarchical model and differential cumulants}. 

Without loss of generality include all
differential operators which are obtained by continued
differentiation. 
Then, \eqref{eqn: hierarchical structure} is mapped exactly to
\begin{equation*}
	D^{\alpha}g(x) = 0, \; \mbox{for all} \; x \in \mathbb{R}^p 
	\Rightarrow D^{\alpha + \gamma} g(x) = 0,\; \mbox{for all} \; x\in \mathbb{R}^p 
\end{equation*}
simply by continued differentiation. 
This bijective mapping from
monomial ideals into differential operators, 
is sometimes referred to as a ``polarity" and within
differential ideal theory has its origins in ``Seidenberg's
differential nullstellensatz" \citep{seidenberg1956}. 
It allows us to map properties of hierarchical models 
in statistics to monomial ideal properties
and vice versa.

One of the main conditions discussed in the theory of hierarchical
models in statistics is the decomposability of a joint density
function into a product of certain marginal probabilities. Simple
conditional probability is a canonical case. Thus with $p=3$ the
conditional independence $X_1 \independent X_2 | X_3$ is represented by
the graph $1 - 3 - 2$. In this case the graph has the model
simplicial complex: $ \mathcal S =\{13,23\}$, 
where, again, we write $\mathcal{S}$ in terms of its maximal cliques. The
Stanley-Reisner ideal is $I_{\mathcal S} = <x_1x_2>$. 

There is a factorization:
\begin{equation*}
	f_{X_1,X_2,X_3}(x_1,x_2,x_3) 
	=
	\frac{f_{X_1,X_3}(x_1,x_3)f_{X_2,X_3}(x_2,x_3)}{f_{X_3}(x_3)}.
\end{equation*}
Decomposable graphical models, discussed below, are a generalization
of this simple case.
There are other cases, however, where one or more factorizations are
associated with the same simplicial complex. An example is the
$4$-cycle: $\mathcal S = \{12,23,34,41\}$ with
The Stanley-Reisner ideal  
$I_\mathcal{S}=<x_1x_3,x_2x_4>$. 
Although this ideal is rather simple from an
algebraic point of view the 4-cycle from a statistical point of view
is rather complex. 
By considering special ideals we obtain
general classes of models, in a subsection \ref{sec: Artinian closure}.

Another issue is that the structure of $\mathcal S$ may suggest
factorizations even when they are problematical. Perhaps the first
such case is the 3-cycle: $\mathcal S = \{12,13,23\}$. 
The Stanley-Reisner ideal is $ I_\mathcal{S}=<x_1x_2x_3>$. 
The maximal clique log-density
representation has no three-way interaction:
\begin{equation*}
	g(x_1,x_2,x_3) 
	= 
	h_{12}(x_1,x_2) + h_{13}(x_1,x_3) + h_{14}(x_1,x_4).
\end{equation*}
This might suggest the factorization
\begin{equation}
	f_{X_1,X_2,X_3}(x_1,x_2,x_3) 
	=
	\frac{f_{X_1,X_2}(x_1,x_2)f_{X_1,X_3}(x_1,x_3)f_{X_2,X_3}(x_2,x_3)}{f_{X_1}(x_1)f_{X_2}(x_2)f_{X_3}(x_3)}
   \label{eqn: three way, no interaction}
\end{equation}
A factorization of this kind is the continuous analogue 
to a perfect three-dimensional table in the discrete case \citep{darroch1962}.
However, except when $X_1,X_2,X_3$ are independent 
we have not been able to provide a standard density for which 
\eqref{eqn: three way, no interaction} holds. 

\subsection{Decomposability and marginality}
Our use of the index set notation makes its
straightforward to define decomposability.
\begin{definition} 
	Let $\mathcal N = \{1,\ldots,p\}$ 
	be the vertex set of a graph $\mathcal G$ and $I, J$  
	vertex sets such that $I \cup J = \mathcal N$.
	Then $\mathcal G$ is decomposable if and only if 
	$I \cap J$ is complete and 
	$I$ forms a maximal clique
	or the subgraph based on $I$ is decomposable 
	and similarly for $J$.
\end{definition}

Under this condition the corresponding hierarchical model has a
factorization
\begin{equation*}
	f_V (x_V) = \frac{ \prod_{J \in C} f_J(x_J)}{\prod_{K\in S} f_K(X_K)},
\end{equation*}
where the numerator on the right hand side corresponds to cliques
and the denominator to separators which arise in the continued
factorization under the definition.

\begin{figure}
	\begin{center}
		\epsfig{file= 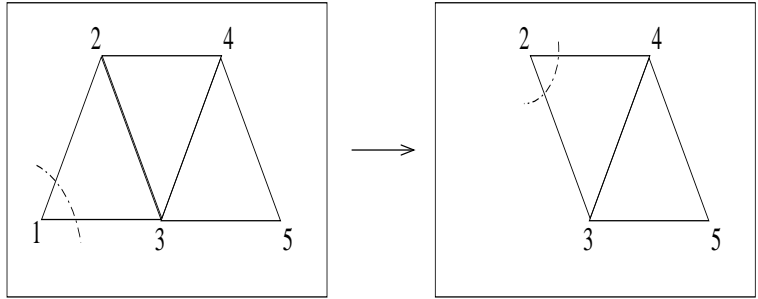,width=10cm,height=4cm}
	\end{center}
	\caption[Decomposable graph]
	{
	   Factorization and marginalisation of a hierarchical model.
	}
	\label{fig: decomposable graph}
\end{figure}
It is important to realize that in order to proceed with the
factorization at each stage a marginalisation step is required. Consider
the simple case based on the simplicial complex $\mathcal{S}=\{123,234,345\}$. 
One choice of factorization at first stage is (with simplified notation):
\begin{equation*}
	f_{12345} = \frac{f_{123}f_{2345}}{f_{23}}
\end{equation*}
and we continue the factorization to give
\begin{equation*}
	f_{12345} 
	= 
	\frac{
				f_{123}f_{234}f_{345}
			}
			{
				f_{23}f_{34}
			}.
\end{equation*}
The process of marginalisation is shown in Figure \ref{fig: decomposable graph}. 
At any stage, we may choose to marginalise with respect to
any variable that is member of just a single clique.
In the first step these are $X_1$ and $X_5$ and suppose we chose to single out $X_1$. 
Once $f_X$ has been integrated with respect to $x_1$,
the marginal model for $X_2,...,X_5$ is obtained.
The removal of a the clique $123$ leads to $X_2$
being exposed and we may continue with $X_2$ or $X_5$ etc.

The Stanley-Reisner ideal 
$I_{\mathcal S} = <x_1x_4, x_1x_5, x_2x_5>$
is an ideal in $k[x_1,x_2,x_3,x_4,x_5]$. 
The factorization of $f_{2345}$ is, however, mapped into the monomial
ideal $<x_2x_5>$ which is an ideal in $k[x_2,x_3,x_4,x_5]$. 
A marginalisation has
allowed us to drop from five dimensions to four. This is clear from
the exponential expression of the model:
\begin{equation*}
	f_{12345}= \exp\left\{h_{123}(x_1,x_2,x_3) + h_{234}(x_2,x_3,x_4) + h_{345}(x_3,x_4,x_5)\right\}.
\end{equation*}
Integrating with respect to $x_1$ we obtain a hierarchical  model
for the marginal joint distribution of $(X_2,X_3,X_4,X_5)$. This
marginalisation is possible because $x_1$ appears only in the single clique
$\{1,2,3\}$.

We have exposed an interesting relationship between the statistical and
algebra formulation: in order to reduce the dimensionality and
obtain the Stanley-Reisner ideal for a reduced set of variables, we
must first perform a marginalisation, which is a non-algebraic operation,
at least, not in general a finite dimensional operation.
We capture this in the following Lemma:
\begin{lemma} \label{lemma: marginalisation}
Whenever a  simplicial complex of hierarchical model has a subset of
vertexes which form a facet of a unique maximal clique (simplex)
then the marginal model obtained by deleting this facet (and its
connections) is valid. Moreover the monomial ideal representation is
obtained by deleting any generators containing the corresponding
variables and is in the ring without these variables.
\end{lemma}

\begin{proof} 
	This follows the lines of the example. If $J$ is
	the subset of vertexes and $K$, with $J \subset K$, is the unique
	maximal clique, then in the exponential expression for the density
	there will be a unique term $\exp(h_K(x_K))$ in which $x_J$ appears.
	Integrating with respect to $x_J$ to obtain the marginal
	distribution for $X_{V \setminus J}$ gives the reduced model. The
	monomial ideal representation follows accordingly.
\end{proof}

\subsection{Artinian closure and polynomial exponential models}\label{sec: Artinian closure}
The terms $h_J(x_J)$ which appear in the hierarchical models have not
been given any special form. In fact it is a main point of this paper
that this is not required to give the monomial ideal equivalence. We
note, again, that we always use square-free monomial ideals.

Certain classes of hierarchical models can, however, be obtained
by imposing further differential conditions. 
The following lemma shows that the log-density is 
polynomial if we impose univariate derivative restriction. 
\begin{lemma}
	If in addition to the differential conditions in Theorem 
	\ref{theorem: hierarchical model and differential cumulants} we impose
	conditions of the form
	\begin{equation}
		\frac{ \partial^{n_i}}{\partial x_i^{n_i}} g(x) = 0,\;
		\text{ for all } 1\leq i \leq p \text{ and } n \in \mathbb{N}^p  \label{eqn: Artinian closure}
   \end{equation}
	the $h$-functions in the corresponding hierarchical
	model are polynomials, in which the degree of $x_i$ does not exceed
	$n_i-1, \; \text{ for all } 1\leq i \leq p$.
\end{lemma}
\begin{proof} 
	Repeated integration with respect to $x_i$ shows that $g$ is indeed a
	polynomial in $x_i$ of degree less than $n_i$, when the other
	variable are fixed. Since this holds for all $1\leq i \leq p$
	the result follows.
\end{proof} 
The simultaneous inclusions of derivative operators with respect to 
one indeterminate in \eqref{eqn: Artinian closure}
constitutes an {\em Artinian closure} of the differential 
version of the Stanley-Reisner ideal $I_\mathcal{S}$.
\begin{example}[BEC density]
   Suppose $X$ is bivariate and we impose 
	the symmetric Artinian closure conditions
	\begin{equation*}
		\frac{ \partial^2}{\partial x_i^2} g(x_1, x_2) = 0,\;
		\text{ for } i = 1,2. 
	\end{equation*} Then integration yields
	\begin{align*}
		g(x_1, x_2) = x_1h_1(x_2)+h_2(x_2) \\
		\intertext{and}
		g(x_1, x_2) = x_2h_3(x_1)+h_4(x_1).
	\end{align*}
	A comparison of these functionals identifies 
	$
		h_1(x_2)=a_3x_2+a_1, \; 
		h_2(x_2)=a_0+a_2x_2 , \;
		h_3(x_1)=a_3x_1+a_2 , \;
		h_4(x_1)=a_1x+a_0,
   $
	for some $a_i \in \mathbb{R}$ for all $1\leq i \leq 4$,
	so that $g(x_1,x_2)$ can be written as 
	\begin{equation}
		g(x_1,x_2)=a_0+a_1x_1+a_2x_2+a_3x_1x_2. \label{eqn: BEC densities}
	\end{equation}
	It can be shown that $X_1$ 
   is distributed exponentially conditional on $X_2=x_2$
   for all $x_2>0$ and vice versa.  
   A distributions with that property is called
   bivariate exponential conditionals (BEC) distribution.
   BEC distributions are completely described by $g$ 
   in the sense that any BEC density is of the 
	form \eqref{eqn: BEC densities}
	\citep{arnold1988}. In particular, the independence 
   case is included, if we force $a_3=0$ by imposing
   the additional restriction  
	\begin{equation*}
		\frac{ \partial^2}{\partial x_1x_2} g(x_1, x_2) = 0.
	\end{equation*}
   This confirms Proposition 
   \ref{prop: independence bivariate case} for this particular example.
\end{example}
The previous example extends readily into higher dimension. 
We call a distribution multivariate exponential conditionals (MEC)
distribution, if $X_j$ is distributed exponentially conditional on $X_i=x_i$ for all
$1 \leq i,j \leq p, i\neq j$. We capture 
the extension to the $p$-dimensional case in the following lemma: 
\begin{lemma}[MEC distributions and Artinian closure]
   The following statements are equivalent:
   \begin{enumerate}
   \item
		A distribution belongs to the class of MEC distributions 
   \item
		$g$ is multi-linear, i.e there exist $2^p$ indices $a_s \in \mathbb{R}$
		such that
		$g=\sum_{s \in \zeta} a_s x^s$, where $\zeta=\{0,1\}^p$ denotes the set of 
      $p$ dimensional binary vectors
   \item
			$\frac{ \partial^2}{\partial x_i^2} g(x) = 0,\;
			\text{ for all } 1\leq i \leq p. 
		   $
   \end{enumerate}
\end{lemma}
\begin{proof} 
	For a proof of $(1) \iff (2)$ see \cite{arnold1988}. 
	The proof of $(2) \iff (3)$ follows the lines of the example.
\end{proof} 

Another case of considerable importance is the Gaussian distribution. Here
$$g(x) = \sum_{K \in \mathcal S} h_K(X_K),$$
and the maximal cliques are of degree two. The latter condition is
partly obtained with an Artinian closure with $n_i = 3,\;i =
1,\ldots, p$. However, more is required. We can guess, from the fact
that for a normal distribution all (ordinary) cumulants of degree
three and above are zero, that if we impose all degree-three differential
cumulant to be zero we obtain polynomial terms of maximum degree 2. 
This is, in fact the correct set of conditions to
make the models terms of degree at most two. In the
$\alpha$-notation the conditions are
\begin{equation*}
	D^{\alpha}g =0,\; \mbox{for all}\; \alpha \in \mathbb{N}^p \;\mbox{with}\; \norm{\alpha}_1=3,
\end{equation*}
which includes the Artinian closure conditions. The
corresponding ideal is generated by all polynomials of degree three.
For a non-singular multivariate Gaussian, we,  of course, require
non-negative definiteness of the degree-two part of the model,
considered as a quadratic form.

The hierarchical model is given by additional restrictions which are
equivalent to removing certain terms of the form $x_ix_j, \; i \neq
j$. This is the same as setting the corresponding $\{ij\}$-th entry
in the inverse covariance matrix (influence matrix) equal to zero.
The removed $x_ix_j$ generate the Stanley-Reisner ideal so that the
zero structure of the influence matrix completely determines the
ideal.

\subsection{Ideal-generated models}
The duality between monomial ideals and hierarchical models
encourages the investigation of the properties of hierarchical
models for different types of ideals. There are some important
properties and features of monomial ideals which may be linked to
the corresponding hierarchical models and 
we mention just a few here in an
attempt to introduce a larger research programme.

We begin with the sub-class of decomposable models.
It is well know from the
statistical literature \citep[see][]{lauritzen1996} that the decomposability
property of the model based on a simplicial complex $\mathcal S$ is
equivalent to the chordal property: there is no chord-less 4-cycle.
Remarkably, the latter is equivalent to a property of
the Stanley-Reisner ideal $I_{\mathcal S}$, namely: that the minimal
free resolution of $I_{\mathcal S}$ be linear(see below for a brief explanation). 
This is a result of  \cite{froeberg1988}, see also \cite{dochtermann2009}.
\cite{petrovic2010} adapt a result of \cite{geiger2006}
to show that $I_{\mathcal S}$, in this case, is generated
in degree 2, that is all its generators have degree 2. 

\begin{theorem} A decomposable graphical model $I_{\mathcal S}$ has a ``2-linear" resolution.  \label{theorem: 2 linear resolution}
\end{theorem}

The term linear refers to the structure of the minimal free
resolution of $I_{\mathcal S}$. In this resolution there are
monomial maps between the stages of the resolution sequence. Linear
means that these maps are linear. As a simple example 
consider again the simplicial 
complex $\mathcal{S}=\{123,234,345\}$ with 
Stanley-Reisner ideal $I_\mathcal{S}=<x_1x_4, x_1x_5, x_2x_5>$
The minimal free resolution of $I_\mathcal{S}$ is given by:
\begin{equation*}
[x_1x_4, x_1x_5, x_2x_5] \begin{array}{c}
                    \tiny{ \left[\begin{array}{rr}
                       -x_5 &  0\\
                       x_4 & -x_2 \\
                       0 & x_1 \\
                     \end{array}\right]}\\
                            \;\;\;\rightarrow\;\;\;\\
                             \\
                          \end{array} 0,
\end{equation*}
and one sees that the map is linear.
By contrast, the  4-cycle is generated in degree 2, but is not
linear:
\begin{equation*}
[x_1x_3,x_2x_4] \begin{array}{c}
            \tiny{\left[\begin{array}{r}
              x_2x_4 \\
              -x_1 x_3 \\
            \end{array}\right]} \\
            \rightarrow \\
             \\
          \end{array} 0,
\end{equation*}
giving a non-linear map.

A special case of 2-linear resolutions are Ferrer ideals. 
A Ferrer ideal $I_{\mathcal S}$ is one in which
the degree-two linear generators can be placed in a table with an
inverse stair-case. Such staircases
arose historically in the study of integer partitions.
As an example take the Stanley-Reisner ideal
\begin{equation*}
	I_{\mathcal S} 
	= 
	<x_1x_6,x_1x_7,x_1x_8,x_2x_6,x_2x_7,x_3x_6,x_3x_7,x_4x_6,x_5x_6>\subseteq k[x_1,...,x_9].
\end{equation*}
The Ferrer table is:
\begin{equation*}
	\begin{array}{c|cccc}
		& 6 & 7 & 8 & 9 \\ \hline
	  1 & x_1x_6 & x_1x_7 & x_1x_8 &  \\
	  2 & x_2x_6 & x_2x_7&  &  \\
	  3 & x_3 x_6 & x_3 x_7 &  &  \\
	  4 & x_4 x_6 & &   &  \\
	  5 & x_5 x_6 &  &  &  \\
	\end{array}
\end{equation*}
Considering the non-empty cells as given by edges this corresponds to a
special type of bi-partite graph between nodes $\{1,2,3,4,5\}$ and
$\{6,7,8,9\}$. \cite{corso2009} show that, among the class of 
bi-partite graphs, Ferrer ideals are indeed
uniquely characterized as having a 2-linear minimal free resolution.

It is straightforward to show that 
the corresponding hierarchical model is decomposable by exhibiting the 
decomposition given by Lemma \ref{lemma: marginalisation}. 
First take two simplices based on the variables defining,
respectively, the rows and columns. In the example these are $J_1 =
12345, J_2=6789$. Then join all nodes corresponding to the
complement of the Ferrer diagram to give:
\begin{equation*}
\begin{array}{c|cccc}
   & 6 & 7 & 8 & 9 \\ \hline
  1 &  &  &  & x_1x_9 \\
  2 &  &  & x_2x_8 & x_2x_9  \\
  3 &   &  & x_3x_8& x_3x_9  \\
  4 &  & x_4 x_7   & x_4 x_8  & x_4x_9 \\
  5 &  & x_5x_7 & x_5x_8 & x_5x_9  \\
\end{array}
\end{equation*}
The maximal cliques are easily seen to be given by a simple rule on this complementary table.
For each non-empty row take the variable which defines that row together with
every other variables for nonempty columns in that row {\em and} all the variables for the rows
below that row. In this example we find, working down the rows, that the maximal cliques
are:
\begin{equation*}
	123459,234589,34589,45789,5789.
\end{equation*}
Note how to this example we can apply Lemma \ref{lemma: marginalisation}, 
by successively stripping off variables in the order: 
$x_1,x_2,x_3,x_4,x_5$. 
The separators are $23459, 34589, 4589, 5789, 789$. The rule provides
a proof of the following.
\begin{theorem} 
Hierarchical models generated by Ferrer ideals are decomposable.
\end{theorem}

As another illustration of the duality between monomial ideals and 
conditional independence structures, we next consider two terminal networks. 
In \cite{saenz2010} the authors apply the theory and construction of minimal free resolutions
to the theory of reliability. One sub-class of these is to networks,
in the classical sense of network reliability. Consider a connected
graph $G = (E,V)$, with two identified nodes called {\em input} and
{\em output}, respectively.  A {\em cut} is a set of edges, which if
removed from the graph disconnects input and output. A {\em path} is
connected set of edges from input to output. A minimal cut is a cut
for which no proper subset is a cut and minimal path is
a path for which no proper subset is a path.

As a simple example consider the network depicted
in Figure \ref{fig: minimal cuts} with input $ = 1$ and
output $=4$ and edges:
$$e_1=1-2,\;e_2=2-4,\;e_3=2-3,\;e_4=1-3,\;e_5=3-4.$$
The minimal cuts are $\{e_1,e_4\},\;\{e_2,e_5\},\;\{e_1,e_3,e_5\},\;
\{e_2,e_3,e_4\}$. If we associate a variable $x_i$ with each edge
$e_i$ then the minimal cuts generate an ideal. In this example we
write
\begin{equation*}
	I_{\mathcal S} = <x_1x_4,x_2x_5,x_1x_3x_5,x_2x_3x_4>.
\end{equation*}
The maximal cliques of $\mathcal S$ for the corresponding model
simplicial complex $\mathcal S$ are $$\{15,24,123,345\}$$
\begin{figure}
	\begin{center}
		\epsfig{file= 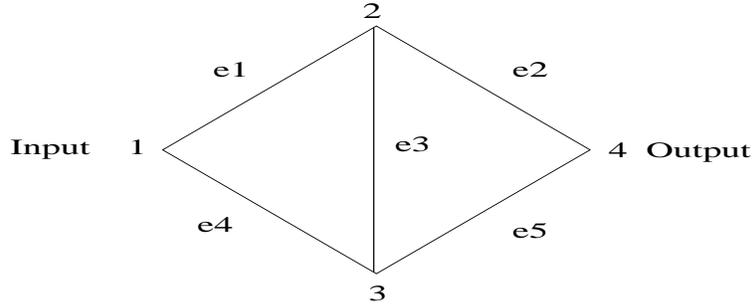,width=10cm,height=4cm}
	\end{center}
	\caption[Terminal network]
	{
      A two-terminal network.
	}
	\label{fig: minimal cuts}
\end{figure}
We could, on the other hand define $I_{\mathcal S^*}$ as being the
collection of all paths on the network. In this case the
$I_{\mathcal S^*}$ is generated by the minimal paths giving:
$$<x_1x_2,x_4x_5, x_1x_3x_5,x_2x_3x_4>,$$
and $S^*$ consists of the complements of the cuts and has maximal
cliques $\{15,24,134,235\}$.

There is a duality between cuts and path models for two-terminal networks:
\begin{lemma} The model simplicial complex $\mathcal S$ based on
the cut ideal $I_{\mathcal S}$ of a two terminal network is formed
from the complement of all paths on the network. Conversely, the
model simplicial complex $\mathcal S^*$ based on the path ideal
$I_{\mathcal S^*}$, is formed from the complement of all cuts.
Moreover: $(\mathcal S^*)^* = S$.
\end{lemma}

For example, the term $15$ of $\mathcal S$ is the complement of the (non-minimal) path $234$ 
and the term $14$ in $\mathcal S^*$ is the complement the (non-minimal) cut $235$.

This duality is a special example of Alexander duality and we omit the proof, 
see \cite{miller2005}, Proposition 1.37.  
The general result says that for a
square-free $\mathcal S$, if we define $\mathcal S^*$ as the
complement of all non-faces of $\mathcal S$, then $(\mathcal S^*)^*
= S$.

It will have been noticed that for this network $S$ and $S^*$ 
are self-dual in the sense that 
the two simplicial complexes have the same structure 
and only differ in the labelling of the vertexes. 
Both models have two separate conditional independence properties. 
Thus for $\mathcal S$ we have
$ X_1 \independent X_4 | (X_2, X_3, X_5)$ and $X_2 \independent X_5 | (X_1, X_3, X_4)$.

\subsection{Geometric constructions} Simplicial complexes are at the
heart of algebraic topology and it is natural to look in that field
for classes of simplicial complexes whose abstract version may be
used to support hierarchical models. We mention briefly one class
here arising from the fast-growing area of persistent homology, see
\cite{edelsbrunner2010}. This class has already been used by
\cite{lunagomez2009} to construct graphical models using so-called Alpha
complexes. We give the construction now. It is to based on the cover
provided by a union of balls in $R^d$,  a construction used by
\cite{edelsbrunner1995} in the context of computational geometry and in
\cite{naiman1992} and \cite{naiman1997} to study Bonferroni bounds in
statistics.

Thus, let $z_1, \ldots, z_p$ be $p$ points in $R^d$ and define the
solid balls with radius $r$ centered at the points:
$$B_i(r) = \{z : ||x-z_i|| \leq r\},\; i = 1, \ldots, p$$
The {\em nerve} of the cover represented by the union of balls is
the simplicial complex $\mathcal S$ derived form the intersections
of the balls, and is called the Alpha complex. It consists of
exactly all index sets $J$ for which $\cap_{i \in J} B_i(r) \neq
\emptyset$.

When the radius, $r$, is small $\mathcal S$ consists of unconnected
vertexes and the hierarchical model gives independence of the $X_1,...,X_p$.
As $r \rightarrow \infty$ there is a value of $r$ at and
beyond which $\cap_{i=1}^p B_i(r) \neq \emptyset$ and $\mathcal S$
consists of a single complete clique. In that case we have a full
hierarchical model. Between these two cases, and depending on the
position of the $z_i$ and the value of $r$ we obtain a rich classes
of simplicial complexes and hence hierarchical models. Some of these
will be decomposable and we refer to the discussion in
\cite{lunagomez2009}.

It is the study of the topology of the nerve as $r$ changes, and in
particular the behavior of its Betti numbers, which drives the area
of persistent homology. A important theoretical and computational
result is that this topology is also that of the reduced simplicial
complex based on the Delauney  complex associated with their Voronoi
diagram. That is to say, for fixed $r$ it is enough, from a
topological (homotopy) viewpoint, to use the sub-complex of the Delauney
complex $\mathcal S^-$ contained in $\mathcal S$. The theory derives
from classical results of \cite{borsuk1948} and \cite{leray1945}. One beautiful fact is
that the Delauney dual complex based on the furthest point Voronoi
diagram \citep{okabe2000}, is obtained by the Alexander duality
mentioned in the last subsection.

In this paper we have concentrated on the correspondence  between
$\mathcal S$ and its Stanley-Reisner ideal $I_{\mathcal S}$. The use
of $I_{\mathcal S}$ is not always explicit in persistent homology
but is implicit in the underlying homology theory: see \cite{saenz2008}
for a thorough investigation, including algorithms. Also, although
the topology of $\mathcal S$ and its reduced Delaunay version
$\mathcal S^-$ is the same, if their actual structure is different
they lead to different hierarchical models. One can also use
non-Euclidean metrics to define the cover and, indeed, work in different
spaces and with other kinds of cover. Notwithstanding these many
interesting technical issues the use of geometric constructions to
define interesting classes of hierarchical model promises to be very
fruitful.

\section{Conclusion}
There are many features and properties of monomial ideals which remain
to be exploited in statistics via the isomorphism discussed in the last section. 
We should mention minimal free resolutions, 
the closely related Hilbert series, Betti numbers 
(including graded and multi-graded versions) 
and Alexander duality. It is pleasing that in the general case
the development of the last section only requires
consideration of square-free ideals, whose theory is
a little easier than the full polynomial case. 
Fast algorithms are available for 
symbolic operations covering all these areas so that as 
further links are made they can be implemented.
We have not covered statistical analysis in this paper. 
Further work is in progress to fit and test 
the zero-cumulant conditions 
$
   D^\alpha g = 0
$
using, for example, kernel methods. 

\bibliographystyle{biometrika}
\bibliography{daniel}

\end{document}